\documentclass{article}

\usepackage[english]{babel}

\usepackage[a4paper, total={6in, 8in}]{geometry}

\usepackage{amsfonts}
\usepackage{amsmath}
\usepackage{amssymb}
\usepackage{amsthm}
\usepackage{bbold}

\usepackage{tikz-cd}
\usepackage{graphicx}
\usepackage{enumitem}

\usepackage{mathtools}

\usepackage{thmtools}
\declaretheorem[numberwithin=section]{theorem}
\declaretheorem[sibling=theorem]{lemma}
\declaretheorem[sibling=theorem]{definition}
\declaretheorem[sibling=theorem]{corollary}
\declaretheorem[sibling=theorem]{proposition}
\declaretheorem[numbered=no, name=Proposition]{proposition*}
\declaretheorem[sibling=theorem, style=remark]{remark}
\declaretheorem[numbered=no, name=Theorem]{theorem*}

\usepackage{stmaryrd}

\usepackage{cite}
\usepackage{hyperref}
\hypersetup{
    colorlinks,
    citecolor=black,
    filecolor=black,
    linkcolor=black,
    urlcolor=black
}

\usepackage[protrusion=true,expansion]{microtype}

\usepackage{lineno}

\newcommand\mc[1]{\mathcal{#1}}

\title{The kernel of the Gysin homomorphism\\ for positive characteristic}
\author{by Claudia Schoemann\thanks{The first author gratefully acknowledges the support by the project “Group schemes, root systems, and
related representations” founded by the European Union - NextGenerationEU through
Romania’s National Recovery and Resilience Plan (PNRR) call no. PNRR-III-C9-2023-
I8, Project CF159/31.07.2023, and coordinated by the Ministry of Research, Innovation
and Digitalization (MCID) of Romania.}\; and Skylar Werner}

\begin{document}

\maketitle

\begin{abstract}
    Let $k$ be an uncountable algebraically closed field of positive characteristic and let $S$ be a smooth projective connected surface over $k$. We extend the theorem on the Gysin kernel from \cite[Theorem 5.1]{PauScho} to also be true over $k$, where it was proved over $\mathbb{C}$. This is done by showing that almost all results still hold true over $k$ via the same argument or by using \'{e}tale base arguments and then using a lift with the Comparison theorems \cite[Theorem 21.1 \& 20.5]{JMilLec} and Tate's Conjecture for finitely generated fields \cite{Tate} and \cite{Zarhin} as needed.
\end{abstract}

\tableofcontents

\bigskip

\newpage

\section{Introduction}
    This paper is based on \cite{PauScho} where $S$ is a smooth projective connected surface over $\mathbb{C}$. Letting $\Sigma$ be the linear system of a very ample divisor $D$ on $S$ with $d = \dim(\Sigma)$, we have the closed embedding $\phi_\Sigma:S \hookrightarrow \mathbb{P}^d$ induced by $\Sigma$. For any closed point $t \in \Sigma = (\mathbb{P}^{d})^\vee$, let $H_t$ be the hyperplane in $\mathbb{P}^d$ defined by $t$ and let $C_t = H_t \cap S$ be the corresponding hyperplane section of $S$. Then we have the closed embedding $r_t:C_t \hookrightarrow S$. Let $\Delta = \{t \in \Sigma = (\mathbb{P}^{d})^\vee : C_t \text{ is singular}\}$ be the discriminant locus of $\Sigma$ and $U = \Sigma \setminus \Delta$ be the set parametrizing smooth curves on $S$. We get the induced map 
    \[
        r_{t*}:H^1(C_t, \mathbb{Z}) \to H^3(S, \mathbb{Z})
    \]
    called the \textit{Gysin homomorphism}\index{Gysin homomorphism!cohomology groups} on cohomology groups. Let 
    \[
        H^1(C_t, \mathbb{Z})_{van} = \ker(r_{t*})
    \]
    be the \textit{vanishing cohomology}\index{vanishing cohomology} of $C_t$. Let 
    \[
        J_t = J(C_t) = J^1(H^1(C_t, \mathbb{Z}))
    \]
    be the \textit{(intermediate) Jacobian}\index{intermediate jacobian} and 
    \[
        B_t = J^1(H^1(C_t, \mathbb{Z})_{van})
    \]
    the abelian subvariety of $J_t$, where we know that $J_t$ is also an abelian variety. We have the \textit{Gysin homomorphism}\index{Gysin homomorphism!Chow groups} 
    \[
        r_{t*}:CH_0(C_t)_{\deg=0} \to CH_0(S)_{\deg=0}
    \]
    on the Chow groups of $0$-cycles of degree zero of $C_t$ and $S$ induced by $r_t$ (see \cite[\S 1.4]{FulInt}), and let the \textit{Gysin kernel}\index{Gysin kernel} be 
    \[
        G_t = \ker(r_{t*}).
    \]
    The following theorem is proved in \cite[Theorem 5.1]{PauScho} using Hodge theory,

    \begin{theorem*}[A theorem on the Gysin kernel]
    \hfill
        \begin{enumerate}[label=(\alph*)]
            \item For every $t \in U$, there is an abelian subvariety $A_t$ of $B_t \subset J_t$ such that
            \[
                G_t = \bigcup_{\text{countable}} \text{translates of } A_t.
            \]

            \item For very general $t \in U$ either, $A_t = 0$ or $A_t = B_t$.
        \end{enumerate}
        where $A_t$ is the unique component containing $0$ and $B_t$ the abelian subvariety associated with\\ $H^1(C_t, \mathbb{Z})_{van}$.
    \end{theorem*}

    Here we let $k$ be an uncountable algebraically closed field of characteristic $p > 0$. Recall that every algebraically closed field is perfect and thus we have an unramified complete discrete valuation ring (CDVR) $W(k)$ such that $char(W(k)) = 0$ and $k$ is its residue field (see \cite[Chapter II \S 5]{SeLocal}). Let $K = Frac(W(k))$ be the fraction field of $W(k)$ and $\overline{K}$ its algebraic closure.

    Recall in general from \cite{PauScho}, that for any smooth projective surface $\mc{S}_0$ over an algebraically closed field $k$, we get the corresponding smooth curve $C_{t_0} = \mc{S}_0 \cap H_{t_0}$ for every $t_0 \in U_0$ where $U_0$ parameterizes all smooth curves on $\mc{S}_0$ and $H_{t_0}$ is the associated hyperplane defined via $t_0$.

    Assume that we have a smooth projective polarized flat family $\mc{S}$ of surfaces over $W(k)$ (see \autoref{sec:LF}). Then in particular, we have a lift from characteristic $p$ to characteristic $0$, i.e. we have the commutative diagram
    \[
        \begin{tikzcd}
          C_{t_0} \arrow[r, hook] \arrow[d, "r_{t_0}"] & \mc{C} \arrow[d] & C_{t_\eta} \arrow[l, hook] \arrow[d, "r_{t_\eta}"] & C_{t_{\overline{\eta}}} \arrow[l] \arrow[d, "r_{t_{\overline{\eta}}}"] \\
          \mc{S}_0 \arrow[r, hook] \arrow[d] & \mc{S} \arrow[d] & \mc{S}_{\eta} \arrow[l, hook] \arrow[d] & \mc{S}_{\overline{\eta}} \arrow[d] \arrow[l] \\
            \{0\} = Spec(k) \arrow[r] & Spec(W(k)) & Spec(K) = \{\eta\} \arrow[l] & Spec(\overline{K}) = \{\overline{\eta}\} \arrow[l]
        \end{tikzcd}
    \]
    where
    \[
        \mc{S}_0 = \{0\} \times_{Spec(W(k))} \mc{S}, \text{ the special fiber}\index{special fiber},
    \]
    \[
        \mc{S}_\eta = \{\eta\} \times_{Spec(W(k))} \mc{S}, \text{ the generic fiber}\index{generic fiber},
    \]
    \[
        \mc{S}_{\overline{\eta}} = \{\overline{\eta}\} \times_{\{\eta\}} \mc{S}_\eta, \text{ the geometric generic fiber}\index{geometric generic fiber}.
    \]

    Recall that the \textit{Lefschetz's principle}\index{Lefschetz's principle} says "there is but one algebraic geometry of characteristic $p$ for each value of $p$" \cite[page 306]{LefPrin}. We even have the stronger statement that roughly says "that algebraic geometry over an arbitrary algebraically closed field of characteristic $0$ is 'the same' as algebraic geometry over $\mathbb{C}$" \cite[\S VI.6]{SilArith}. J. Barwise and P. Eklof proved Lefschetz's Principle meta-mathematically in \cite{BarEkl}.

    Thus, doing algebraic geometry over $\overline{K}$ is 'the same' as doing algebraic geometry over $\mathbb{C}$. We then use the two comparison theorems from \cite[Theorem 20.5 and Theorem 21.1]{JMilLec} to switch from working with singular cohomology to  \'{e}tale cohomology and then working with the \'{e}tale cohomology over the special fiber:

    \begin{theorem*}[Theorem 21.1, Artin's Comparison Theorem]
        Let $X$ be a nonsingular variety over $\mathbb{C}$. For any finite abelian group $\Lambda$ and $r \ge 0, H^r_{\acute{e}t}(X, \Lambda) \cong H^r(X, \Lambda)$, where $H^r(X, \Lambda)$ is the classic singular cohomology.
    \end{theorem*}

    \begin{theorem*}[Theorem 20.5]
        Suppose that a variety $X_0$ over an algebraically closed field $k$ of characteristic $p \neq 0$ can be lifted to a variety $X_1$ over a fixed field $K$ of characteristic zero. For any abelian group $\Lambda$,
        \[
            H^r(X_0, \Lambda) \cong H^r(X_{1, \overline{K}}, \Lambda),
        \]
        where $X_{1, \overline{K}}$ is the variety over the algebraic closure of $K$.
    \end{theorem*}

    In the art of proving part (a) of a theorem on the Gysin kernel, we show that all pre-lemmas and propositions about the countability of varieties, regularity and representability of the Chow group $CH_0(X)_{\deg=0}$ of zero cycles of degree zero still hold true in our general setting. Then when we need to use cohomological arguments, we use the lift and comparison theorems on the previous results showed initially over $\mathbb{C}$.

    For part (b), we are highly inspired and follow closely \cite[\S 6]{BanGul} to show that the (\'{e}tale) fundamental group $\pi_1(U', \overline{\eta})$ based on the geometric generic point acts continuously on the cohomology groups that we will be working with. In particular, we will show that we can always work with a Lefschetz pencil and analyze the vanishing cycles which are irreducible (see \cite{Freitag}).
    
\section{A Brief History}
    This section goes over a very brief history of this problem and thus can be skipped without any consequence to the rest of the paper.

    It has been classically known that in the theory of algebraic curves there is a map $Sym^n(C) \to J(C)$ from the \textit{$n$-th symmetric product}\index{symmetric product} of the curve $C$ into the \textit{Jacobian Variety}\index{jacobian variety} where the fibers are projective spaces, i.e. representing the linear systems of degree $n$. As usual in mathematics, when adding on an extra dimension, i.e. for algebraic surfaces, the whole system becomes a lot more difficult. Naturally we have the analogous map $Sym^n(S) \to Alb(S)$ where $Alb(S)$ is the \textit{Albanese variety}\index{albanese variety}. The fibers are irreducible and regular when $n$ is large enough (see \cite{RuledSurf}). However, in the 1960's, a long open problem was if the fibers are rational. One of the first leading researchers on the problem was Severi \cite{Severi}. After Severi, Arthur Mattuck proved the following

    \begin{theorem*}
        Let $S$ be a complete non-singular surface in characteristic zero, and let $q = \dim(Alb(S))$. If for some $n > q$ the general fiber of the morphism $Sym^n(S) \to Alb(S)$ is a rational variety, then $V$ is a ruled surface.
    \end{theorem*}

    \begin{proof}
        see \cite{RuledSurf}
    \end{proof}

    Very similarly and independently, David Mumford in the article \cite{Mum} has the following

    \begin{theorem*}
        Let $S$ be a non-singular complex algebraic surface. For all $f:S \to Sym^n(S)$ such that all the $0$-cycles $f(s),\, s \in S$, are rationally equivalent, it follows that $\eta_f = 0$ where $\eta_f$ is a $2$-form in $\Gamma(S, \Omega_S^2)$.
    \end{theorem*}

    As we can see in Mumford's theorem, differential $2$-forms were used. Similarly in the proof of Mattuck's theorem, differential $2$-forms were equally used. At a quick glance at this paper, we do not use differentials at all. In Mumford \cite{Mum}, lemma 3 states that $Sym^{n,n}(S)$ contains a countable set of closed subvarieties $Z_i$ such that if $(A, B) \in Sym^{n,n}(S)$, then $A \sim_{rat} B \iff (A, B) \in \bigcup_{i = 1}^\infty Z_i$. From this, A. A. Ro\u{i}tman produced two articles generalizing this subject to any uncountable algebraically closed field of characteristic $0$ (see \cite{RoitGamma} and \cite{RoitRat}). In particular, we have proposition $2$ from \cite{RoitRat},

    \begin{proposition*}
        Let $Alb(X)$ be the Albanese variety of a non-singular irreducible projective variety $X$ over an uncountable algebraically closed field of characteristic $0$. Then the canonical epimorphism $CH_0(X) \to Alb(X)$ induces an epimorphism on the points of finite order:
        \[
            CH_0(X)_{tor} \to Alb(X)_{tor} \to 0.
        \]
    \end{proposition*}

    Even though in his article he called it a proposition, this has become an extreme topic point in research and thus people call it Ro\u{i}tman's theorem. A decade later, Milne abstracted Ro\u{i}tman's theorem to work over fields of characteristic $p\neq 0$,

    \begin{theorem*}
        Let $k$ have characteristic $p \neq 0$, then $CH_0(X)(p) \to Alb(X)(p)$ is an isomorphism where $CH_0(X)(p)$ and $Alb(X)(p)$ are the $p$-primary components.
    \end{theorem*}

    \begin{proof}
        See \cite{MilRoit}.
    \end{proof}

    From here, Spencer Bloch started investigating the relation between algebraic cycles and higher $K$-theory, which he called higher Chow groups (see \cite{BloAlgK}). This has been a very successful route of research, but for what we focus on, we will not talk about $K$-theory at all. Outside of higher Chow groups, Bloch was able to generalize Milne's theorem,

    \begin{theorem*}
        Let $X$ be a smooth projective variety over an algebraically closed field $k$. Then the map $CH_0(X) \to Alb(X)$ induces an isomorphism on torsion prime to the characteristic of $k$.
    \end{theorem*}

    \begin{proof}
        See \cite[Theorem 5.1]{SBloch}.
    \end{proof}

    Except for the few generalizations, most of these results are over $\mathbb{C}$. Because of this, many people started using Hodge theory to obtain results, in particular, studying the monodromy action (see \cite{CVoi2}). Since the method of Hodge theory only works over the complex numbers, in their recent paper \cite{BanGul}, K. Banerjee and V. Guletski\u{i} changed the approach of the problem by using \'{E}tale theory. But on this note, they still worked over an uncountable algebraically closed field of characteristic $0$. They do this in order to 
    make three assumptions. After this article, in \cite{PauScho} the authors go back to working over $\mathbb{C}$ and prove in detail that the three assumptions from Banerjee and Guletski\u{i} are true in their case, then ending their paper with a theorem on the Gysin kernel.

\section{Relation between algebraic, rational and homological equivalence}\label{relation}
    Let $X$ be a smooth projective reduced scheme of dimension $n$ over an algebraically closed field $k$ of arbitrary characteristic. For the definitions and proofs of the lemmas below see \cite[\S 2]{PauScho}.
    
    \subsection{Rational, Algebraic and Homological Equivalence of \texorpdfstring{$0$}{Lg}-cycles}

        Recall that the \textit{cycle map} is a homomorphism of graded groups which doubles degrees, $cl_{X}:Z^{*}(X) \rightarrow H^{*}(X), Z \mapsto i_{*}(1_{Z})$, where $i_{*}$ is the \textit{Gysin map} $H^{0}(Z, \Lambda) \rightarrow H^{2r}(X,\Lambda(r))$ and $1_{Z} \in H^{0}(Z,\Lambda)$ is the identity element of the ring $H^{*}(Z)$ (see \cite[Proposition 6.5]{JMilEtal}).
    
        \begin{lemma}\label{lem2.44}
            Assume in addition that $X$ is connected and let $Z_0(X)_{\deg=0} \subset Z_0(X)$ be the group of $0$-cycles of degree $0$. Then $Z_0(X)_{\hom} = Z_0(X)_{\deg = 0}$.
        \end{lemma}
        
        \begin{proof}
            Since $X$ is connected, we have $H_{\acute{e}t}^{2n}(X, \mathbb{Z}_p) = \mathbb{Z}_p(-n)$ (see \cite[Chapter 3 Proposition 4.5]{Freitag}). We know that for any sheaf $F$ of $\mathbb{Z}/p\mathbb{Z}$-modules  the $r$-th Tate twist is $F(r) = F \otimes \mathbb{Z}/p\mathbb{Z}(r)$ and thus $F(r)$ is locally (non-canonically) isomorphic to $F$ (see \cite[page 163]{JMilEtal}). If $X$ is a variety over a separably closed field, then $F(r) \cong F$ non-canonically. Taking $F = \mathbb{Z}/p\mathbb{Z}$, we get that $\mathbb{Z}/p\mathbb{Z}(r) \cong \mathbb{Z}/p\mathbb{Z}$. Taking the limit, we get $\mathbb{Z}_p(r) \cong \mathbb{Z}_p$. Then $Z_0(X)_{\hom} = \ker(cl:Z_0(X) \to H_{\acute{e}t}^{2n}(X, \mathbb{Z}_p)) = \ker(cl:Z_0(X) \to \mathbb{Z}_p(-n)) = \ker(cl:Z_0(X) \to \mathbb{Z}_p ) = \ker(\deg:Z_0(X) \to \mathbb{Z}) = Z_0(X)_{\deg=0}$.
        \end{proof}
        
        \begin{lemma}\label{lem2.45}
            If $C$ is an integral (reduced and irreducible) smooth curve and $P_1, P_2 \in C$ are two points in $C$, then $P_1 \sim_{\text{alg}} P_2$.
        \end{lemma}
        
        \begin{lemma}\label{lem2.46}
            If there exists a connected curve $C$ such that its components are smooth and integral and $P$ and $Q$ are two points of $C$, then $P \sim_{\text{alg}} Q$.
        \end{lemma}
        
        \begin{proposition}\label{prop2.47}
            If $X$ is connected, then $Z_0(X)_{\hom} = Z_0(X)_{\text{alg}}$.
        \end{proposition}
        
        \begin{proposition}\label{prop2.48}
            Let $X$ be a smooth projective variety of dimension $n$ over an algebraically closed field $k$ of arbitrary characteristic. Then $CH_0(X)_{\deg=0} = CH_0(X)_{\hom} = A_0(X)$.
        \end{proposition}
        
        \begin{lemma}
            Let $C$ be a smooth projective curve over an algebraically closed field $k$ of arbitrary characteristic. Then $CH_0(X)_{\deg=0} = CH_0(C)_{\hom} = A_0(C)$.
        \end{lemma}
        
\section{The Lift and Specialization map}

    \subsection{Lift via local fields}
    \label{sec:LF}
    
        Recall that the formal power series (and thus the formal Laurent series) $\mathbb{F}_{p}[[x]]$ of the finite field $\mathbb{F}_{p}$ of $p$ elements in one indeterminate is uncountable. To ensure that our field is algebraically closed, we artificially take the algebraic closure, i.e. $\overline{\mathbb{F}_p((x))}$. Thus, $\overline{\mathbb{F}_p((x))}$ is an uncountable algebraically closed field of characteristic $p$. Naturally we get an infinite set of fields by having more variables, i.e. $\overline{\mathbb{F}_p((x_1, ..., x_n))}$.

        We know via the study of local fields that we always have a lift from characteristic $p$ to characteristic $0$ when we are working over an uncountable algebraically closed field $k$ of positive characteristic. This is given by the theorems below whose  proofs can be found in 
        \cite{SeLocal}.
    
        \begin{theorem}(Theorem 3)
            For every perfect field $k$ of characteristic $p$ there exists a complete discrete valuation ring and only one (up to unique isomorphism) which is unramified and has $k$ as its residue field.
        \end{theorem}
    
        We denote the completed discrete valuation ring (CDVR) in the above theorem as $W(k)$. It 
        can be explicitly constructed and is called the \textit{ring of Witt vectors}\index{Witt Vectors}\index{ring of Witt vectors} (see \cite[\S 6]{SeLocal}). We know that the ring of Witt vectors is the correct CDVR since we have the theorem
    
        \begin{theorem}(Theorem 8)
            If $k$ is a perfect ring of characteristic $p$, $W(k)$ is a strict $p$-ring with residue ring $k$.
        \end{theorem}
    
        Let $\mc{S}$ be a smooth projective polarized flat family of smooth projective connected surfaces over $W(k)$. Then we get the following commutative diagram.
        
        \[
            \begin{tikzcd}
                C_{t_0} \arrow[r, hook] \arrow[d, "r_{t_0}"] & \mc{C} \arrow[d] & C_{t_\eta} \arrow[l, hook] \arrow[d, "r_{t_\eta}"] & C_{t_{\overline{\eta}}} \arrow[l] \arrow[d, "r_{t_{\overline{\eta}}}"] \\
                {S}_0 \arrow[r, hook] \arrow[d] & \mc{S} \arrow[d] & {S}_{\eta} \arrow[l, hook] \arrow[d] & {S}_{\overline{\eta}} \arrow[d] \arrow[l] \\
                \{0\} = Spec(k) \arrow[r] & Spec(W(k)) & Spec(K) = \{\eta\} \arrow[l] & Spec(\overline{K}) = \{\overline{\eta}\} \arrow[l]
            \end{tikzcd}
        \]
        where
        \[
          S_0 = \{0\} \times_{Spec(W(k))} \mc{S}, \text{ the special fiber}\index{special fiber},
        \]
        \[
          S_\eta = \{\eta\} \times_{Spec(W(k))} \mc{S}, \text{ the generic fiber}\index{generic fiber},
        \]
        \[
          S_{\overline{\eta}} = \{\overline{\eta}\} \times_{\{\eta\}} {S}_\eta, \text{ the geometric generic fiber}\index{geometric generic fiber}.
        \]

        Since $\mc{S}_0, \mc{S}_{\eta}$ and $\mc{S}_{\overline{\eta}}$ are smooth projective connected surfaces, we can define linear systems of very ample divisors $D_0, D_{\eta}$ and $D_{\overline{\eta}}$, respectively. Since the family is polarized (see \cite[\S 3.3.3]{Sernesi}), the lift $\mc{S}$ is equipped with a relatively very ample line bundle $\mathcal{L}$ such that its restriction to the special fiber is $\mathcal{L}|_{\mc{S}_0} \cong \mathcal{O}_{\mc{S}_0}(D_0)$.

        Because very ampleness is an open condition on the base \cite[Théorème 4.7.1]{EGAIII}, and $\mathcal{L}|_{\mc{S}_0}$ is very ample, the line bundle $\mathcal{L}$ is relatively very ample. Consequently, $\mathcal{L}$ determines a global closed embedding of the family $\Phi_{\mathcal{L}}:\mc{S} \hookrightarrow \mathbb{P}_{W(k)}^d$ where $d = h^0(\mathcal{L}|_{\mc{S}_{\eta}})-1$. Restricting this embedding to the fibers shows that the linear systems corresponding to $D_0, D_{\eta}$ and $D_{\overline{\eta}}$ all have the same dimension $d$. Thus, we obtain the closed embeddings into $\mathbb{P}^d$ induced by these linear systems.

        \[
            \phi_{\Sigma_0}:S_0 \hookrightarrow \mathbb{P}^d,
        \]
        \[
            \phi_{\Sigma_\eta}:S_\eta \hookrightarrow \mathbb{P}^d,
        \]
        \[
            \phi_{\Sigma_{\overline{\eta}}}:S_{\overline{\eta}} \hookrightarrow \mathbb{P}^d.
        \]
        
        For any closed point $t_0 \in \Sigma_0 = (\mathbb{P}^{d})^\vee$, let $H_{t_0}$ be the hyperplane in $\mathbb{P}^d$ defined by $t_0$, let $C_{t_0} = H_{t_0} \cap S_0$ be the corresponding hyperplane section of ${S}_0$ and let

        \[
            r_{t_0}:C_{t_0} \hookrightarrow S_0
        \]
        be the closed embedding. Similarly, we can do the same for any closed points $t_\eta \in \Sigma_\eta$ and $t_{\overline{\eta}} \in \Sigma_{\overline{\eta}}$ giving the closed embeddings
        \[
            r_{t_\eta}:C_{t_\eta} \hookrightarrow S_\eta,
        \]
        \[
            r_{t_{\overline{\eta}}}:C_{t_{\overline{\eta}}} \hookrightarrow S_{\overline{\eta}}.
        \]
        Let $\Delta_0, \Delta_\eta$ and $\Delta_{\overline{\eta}}$ be the \textit{discriminant locus}\index{discriminant locus} of their linear systems, i.e.
        \[
            \Delta_0 = \{t_0 \in \Sigma_0 = (\mathbb{P}^{d})^\vee : C_{t_0} \text{ is singular}\},
        \]
        \[
            \Delta_\eta = \{t_\eta \in \Sigma_\eta = (\mathbb{P}^{d})^\vee : C_{t_\eta} \text{ is singular}\},
        \]\[
            \Delta_{\overline{\eta}} = \{t_{\overline{\eta}} \in \Sigma_{\overline{\eta}} = (\mathbb{P}^{d})^\vee : C_{t_{\overline{\eta}}} \text{ is singular}\}.
        \]
        Let $U_0 = \Sigma_0 \setminus \Delta_0,\, U_\eta = \Sigma_\eta \setminus \Delta_\eta$ and $U_{\overline{\eta}} = \Sigma_{\overline{\eta}} \setminus \Delta_{\overline{\eta}}$ be the complement of the discriminant locus parametrizing smooth curves on their respective surfaces. Then we get the respective \textit{Gysin homomorphisms}\index{Gysin homomorphism!cohomology groups} on cohomology groups
        \[
            r_{t_0*}:H_{\acute{e}t}^1(C_{t_0}, \Lambda) \to H_{\acute{e}t}^3(S_0, \Lambda(1)),
        \]
        \[
            r_{t_\eta*}:H_{\acute{e}t}^1(C_{t_\eta}, \Lambda) \to H_{\acute{e}t}^3(S_\eta, \Lambda(1)),
        \]
        \[
            r_{t_{\overline{\eta}}*}:H_{\acute{e}t}^1(C_{t_{\overline{\eta}}}, \Lambda) \to H_{\acute{e}t}^3(S_{\overline{\eta}}, \Lambda(1)),
        \]
        where $\Lambda$ is a finite abelian group (see \cite[\S 24 The Gysin map]{JMilLec}). For each Gysin homomorphism, define the \textit{vanishing cohomology}\index{vanishing cohomology} to be the kernel, i.e.
        \[
            H_{\acute{e}t}^1(C_{t_0}, \Lambda)_{van} = \ker(r_{t_0*}),
        \]
        \[
            H_{\acute{e}t}^1(C_{t_\eta}, \Lambda)_{van} = \ker(r_{t_\eta*}),
        \]
        \[
            H_{\acute{e}t}^1(C_{t_{\overline{\eta}}}, \Lambda)_{van} = \ker(r_{t_{\overline{\eta}}*}).
        \]
        
        Let $J_t = J(C_t)$ be the \textit{Jacobian}\index{jacobian}, i.e. the Jacobian of the curve $C_t$. Thus in our setting, we have the Jacobians
        \[
            J_{t_0} = J(C_{t_0}),
        \]
        \[
            J_{t_\eta} = J(C_{t_\eta})
        \]
        \[
            J_{t_{\overline{\eta}}} = J(C_{t_{\overline{\eta}}}).
        \]
        The Jacobian of the vanishing cohomology is denoted by $B_t$, that is, we have
        \[
            B_{t_0} = J(H_{\acute{e}t}^1(C_{t_0}, \Lambda)_{van}),
        \]
        \[
            B_{t_\eta} = J(H_{\acute{e}t}^1(C_{t_\eta}, \Lambda)_{van}),
        \]
        \[
            B_{t_{\overline{\eta}}} = J(H_{\acute{e}t}^1(C_{t_{\overline{\eta}}}, \Lambda)_{van}).
        \]
        It is known that we can push-forward cycles (see \cite[Theorem 1.4]{FulInt}) and that the degree of $0$-cycles are preserved under this push-forward (see \cite[Definition 1.4]{FulInt}). Thus, we get the \textit{Gysin homomorphism}\index{Gysin homomorphism!Chow groups} on the Chow groups of 0-cycles of degree zero,
        \[
            r_{t_0*}:CH_0(C_{t_0})_{\deg=0} \to CH_0(S_0)_{\deg=0},
        \]
        \[
            r_{t_\eta*}:CH_0(C_{t_\eta})_{\deg=0} \to CH_0(S_\eta)_{\deg=0},
        \]
        \[
            r_{t_{\overline{\eta}*}}:CH_0(C_{t_{\overline{\eta}}})_{\deg=0} \to CH_0(S_{\overline{\eta}})_{\deg=0}.
        \]
        The \textit{Gysin kernel}\index{Gysin kernel} is denoted as
        \[
            G_{t_0} = \ker(r_{t_0*}),
        \]
        \[
            G_{t_\eta} = \ker(r_{t_\eta*}),
        \]
        \[
            G_{t_{\overline{\eta}}} = \ker(r_{t_{\overline{\eta}}*}).
        \]
        
        \subsection{The specialization homomorphism}
        While there is no morphism of schemes from the generic fiber to the special fiber, the flatness of the family allows us to define a \textit{specialization homomorphism} on the level of cycles. Following \cite[\S 20.3]{FulInt}, for any cycle $Z$ on the generic fiber $S_{\overline{\eta}}$, we define its specialization $sp(Z)$ by taking the Zariski closure $\overline{Z}$ in the total family $\mc{S}$ and restricting to the special fiber. Since the family is flat, this operation respects rational equivalence and preserves the degree of $0$-cycles. This yields the \textit{specialization map} on Chow groups:
        \[
            sp_{\mc{S}}: CH_0(S_{\overline{\eta}})_{\deg=0} \to CH_0({S}_0)_{\deg=0}.
        \]

        A similar homomorphism $sp_{\mc{C}}$ exists for the family of curves.

        To relate this to the symmetric products, we consider the relative symmetric product of the family, $\mathcal{S}ym^d(\mathcal{C}/W(k))$. As the construction of the symmetric product commutes with base change for flat families (geometric quotients by finite groups commute with any flat base change \cite[Expos\'{e} V \S 1]{SGAI}), the fibers of this relative scheme are isomorphic to $Sym^d(C_{\overline{\eta}})$ and $Sym^d(C_0)$, respectively. Since $\mathcal{S}ym^d(\mathcal{C})$ is proper over $W(k)$, we apply the valuative criterion of properness (see \cite[Theorem 4.7]{Hart}). Any geometric point $P_{\overline{\eta}} \in Sym^d(C_{\overline{\eta}})(\overline{K})$ extends to a unique section over $W(k)$, which intersects the special fiber at a unique point $P_0 \in Sym^d(C_0)(k)$. This defines a specialization map on the set of geometric points:

        \[
            sp_{Sym}: Sym^d(C_{\overline{\eta}})(\overline{K}) \to Sym^d(C_0)(k).
            \]

            This map is compatible with the specialization on Chow groups via the difference map $\theta$. Consequently, we obtain the following commutative diagram relating the fibers:

        \[
            \begin{tikzcd}
                {Sym^{d,d}(C_{t_{\overline{\eta}}})(\overline{K})} \arrow[dd, "{\theta_{d,d}^{C_{t_{\overline{\eta}}}}}"] \arrow[rr, "sp_{_{Sym}}", dashed] \arrow[rd, "{r_{t_{\overline{\eta}}, Sym}}"] & & {Sym^{d,d}(C_{t_0})(k)} \arrow[rd, "{r_{t_0, Sym}}"] \arrow[dd, "{\theta_{d,d}^{C_{t_0}}}" near start] & \\
                                                                                                                & {Sym^{d,d}(S_{\overline{\eta}})(\overline{K})} \arrow[rr, "sp_{_{Sym}}" near start, dashed] \arrow[dd, "{\theta_{d,d}^{{S}_{\overline{\eta}}}}" near start] & & {Sym^{d,d}({S}_0)(k)} \arrow[dd, "{\theta_{d,d}^{{S}_0}}"] \\
                CH_0(C_{t_{\overline{\eta}}})_{\deg=0} \arrow[dd, "alb_{C_{t_{\overline{\eta}}}}"] \arrow[rr, "sp_{\mc{C}}" near start] \arrow[rd, "r_{t_{\overline{\eta}}*}"] & & CH_0(C_{t_0})_{\deg=0} \arrow[rd, "r_{t_0*}"] \arrow[dd, "alb_{C_{t_0}}" near start] & \\
                                                                                                                & CH_0(S_{\overline{\eta}})_{\deg=0} \arrow[rr, "sp_{_{\mc{S}}}" near start] & & CH_0({S}_0)_{\deg=0} \\
                J_{t_{\overline{\eta}}} \arrow[rr, "sp_{_J}"] & & J_{t_0} &
            \end{tikzcd}
        \]
        
        where the dashed arrows represent the set-theoretic specialization of points described above, and the solid horizontal arrows represent the homomorphisms on Chow groups and Abelian varieties.
        
\section{Countability Lemmas}
    The proofs can be found in \cite[\S 4]{BanGul}. They work in this new general case due to Noether normalization lemma \cite{ENoeth}, i.e. since $k$ is uncountable, we get an uncountable amount of hyperplanes.

    \begin{lemma}\label{CountContradiction}
        Let $V$ be an irreducible quasi-projective variety over an uncountable field $k$ of arbitrary characteristic. Then $V$ cannot be written as a countable union of its Zariski closed subsets, each of which is not the whole of $V$.
    \end{lemma}
    
    \begin{definition}[Irredundant countable union\index{irredundant}]
        A countable union $V = \bigcup_{n \in \mathbb{N}} V_n$ of algebraic varieties will be called \textit{irredundant} if $V_n$ is irreducible for each $n$ and $V_m \not\subset V_n$ for $m \neq n$. If $V$ is an irredundant decomposition, then the sets $V_n$ are called c-components of $V$.
    \end{definition}
    
    \begin{lemma}\label{uniqueIrred}
        Let $V$ be a countable union of algebraic varieties over an uncountable algebraically closed ground field $k$ of arbitrary characteristic. Then $V$ admits an irredundant decomposition, and such an irredundant decomposition is unique.
    \end{lemma}
    
    \begin{lemma}\label{irredundantAbelian}
        Let $A$ be an abelian variety over an uncountable field $k$ of arbitrary characteristic, and let $K$ be a subgroup which can be represented as a countable union of Zariski closed subsets in $A$. Then the irredundant decomposition of $K$ contains a unique irreducible component passing through $0$, and this component is an abelian subvariety in $A$.
    \end{lemma}
    
\section{Regular maps and Representability of \texorpdfstring{$CH_0(X)_{\deg=0}$}{Lg}}
    Let $X$ be a nonsingular projective variety of dimension $n$ that in this section is defined over an uncountable algebraically closed field of arbitrary characteristic.

    Several results in this section are classical when the base field is $\mathbb{C}$ and are proved in this setting in \cite{RoitRat} and \cite{CVoi2}. Our aim here is not to reproduce these arguments, but to explain why they extend without change to smooth projective varieties defined over an uncountable algebraically closed field of arbitrary characteristic.

    \subsection{Regular maps}
    
        \begin{definition}[c-closed and c-open\index{c-closed}\index{c-open}]
            \label{copen}
            A subset of an integral algebraic scheme $T$ which is the union of a countable number of closed subsets is called a \textit{c-closed}\index{c-open} subset and the complement of a \textit{c-closed}, i.e. intersections of a countable number of open subsets, is called a c-open subset.
        \end{definition}
        
        \begin{definition}[Difference map\index{difference map}]\label{diffMap}
          The set-theoretic map $\theta_{d_{1},d_{2}}^{X}:Sym^{d_{1}}(X) \times Sym^{d_{2}}(X) \to CH_{0}(X), (A,B) \mapsto [A-B]$, where $[A-B]$ is the class of the cycle $A-B$ modulo rational equivalence, will be called the \textit{difference map}.
        \end{definition}
        
        Let $W^{d_{1},d_{2}} = \{(A,B;C,D) \in Sym^{d_{1}, d_{2}}(X) \times Sym^{d_{1}, d_{2}}(X) : \theta_{d_{1},d_{2}}^{X} (A,B) = \theta_{d_{1}, d_{2}}^{X}(C,D) \}$ $= \{(A,B;C,D) \in Sym^{d_{1}, d_{2}}(X) \times Sym^{d_{1}, d_{2}}(X) : (A-B) \sim_{\text{rat}} (C-D) \}$ be the subset of $Sym^{d_1, d_2}(X) \times Sym^{d_1, d_2}(X)$ defining the rational equivalence on $Sym^{d_1, d_2}(X)$.
        
        \begin{lemma}\label{Wdcclosed}
            The subset $W^{d_1, d_2}$ is c-closed.
        \end{lemma}
        
        \begin{proof}
            The fact that the subset $W^{d_1,d_2}$ is a countable union of closed subsets is proved over $\mathbb{C}$ in \cite[Theorem~1]{RoitRat}. We briefly explain why the same argument applies over an arbitrary uncountable algebraically closed field.

            Roitman’s proof is entirely algebraic. The key step is to express the relation of rational equivalence between two effective cycles of fixed degree as the image, under a projection morphism, of a parameter space of algebraic maps from pointed curves into symmetric powers of $X$. More precisely, $W^{d_1,d_2}$ is written as a countable union of subsets, each of which is the image of a quasiprojective variety constructed as a fiber product involving a Hom-scheme of morphisms from a curve to a symmetric product of $X$. The proof then shows that each such subset is constructible and, in fact, closed.

            All ingredients involved in this construction - Hom-schemes of morphisms between projective varieties, symmetric products, fiber products and images of morphisms of finite type - are available over any algebraically closed field. The argument uses 
            general properties of morphisms of finite type and the description of rational equivalence via algebraic families parametrized by curves. No transcendental methods or characteristic zero assumptions are applied.

            Consequently, the proof of \cite[Theorem~1]{RoitRat} carries over verbatim to our setting, and $W^{d_1,d_2}$ is c-closed over an arbitrary uncountable algebraically closed field.
        \end{proof}
        
        \begin{definition}[Regular map into $CH_0(X)$\index{regular map}]
            A set-theoretic map $\kappa:Z \to CH_0(X)$ of an algebraic variety $Z$ into the Chow group of $0$-cycles $CH_0(X)$ will be called \textit{regular} if there exists a commutative diagram (set-theoretically)
            \[
                \begin{tikzcd}
                    Y \arrow[r, "f"] \arrow[d, "g"] & Sym^{d_1,d_2}(X) \arrow[d, "\theta^X_{d_1, d_2}"] \\
                    Z \arrow[r, "\kappa"] & CH_0(X)
                \end{tikzcd}
            \]
            where $f$ is a regular map and $g$ is an epimorphism which is also a regular map.

            If $Z$ is projective (irreducible), then its image $\kappa(z) \subset CH_0(X)$ under a regular map will be called a \textit{closed} (irreducible) subset.
        \end{definition}
        
        Equivalently, set-theoretic regular maps into $CH_0(X)$ can be defined as follows \cite[Lemma 4]{RoitRat}.
        
        \begin{lemma}[Alternative definition of a regular map into $CH_0(X)$]
            The set-theoretic map $\kappa:Z \to CH_0(X)$ is regular if and only if for any integers $d_1$ and $d_2$ the subset $W_{\kappa, \theta_{d_1, d_2}^X} = \{(z, A, B) \in Z \times Sym^{d_1, d_2}(X) : \kappa(z) = \theta_{d_1, d_2}^X(A, B)\} = Z \times_{CH_0(X)} Sym^{d_1,d_2}(X)$ is c-closed.
        \end{lemma}
        
        \begin{lemma}\label{reg1}
            The map $\theta_{d_1, d_2}^X:Sym^{d_1, d_2}(X) \to CH_0(X)$ is regular.
        \end{lemma}
        
        \begin{lemma}\label{reg2}
            Let $\kappa:Z \to CH_0(X)_{\deg=0}$ be a regular map and let
            \[ alb_X:CH_0(X)_{\deg=0} \to Alb(X) \]
            be the Albanese map. Then the composite map $alb_X \circ \kappa:Z \to Alb(X)$ is a regular map of algebraic varieties.
        \end{lemma}
        
    \subsection{Representable}
        \begin{definition}[Representability\index{representable}]
            $CH_0(X)_{\deg=0}$ is \textit{representable} if the natural map $\theta_d^X:Sym^d(X) \times Sym^d(X) \to CH_0(X)_{\deg=0}$ is surjective for sufficiently large $d$ (see \cite[Def. 10.6]{CVoi2}).
        \end{definition}
        
        \begin{remark}
            In words, $CH_0(X)_{\deg=0}$ is representable if and only if $CH_0(X)_{\deg=0}$ is finite-dimensional.
        \end{remark}

        We will study the nature of what happens when $CH_0(X)_{\deg=0}$ is finite-dimensional and write an alternative definition of representability (Corollary 6.11 below). In particular, we have this very strong statement that hints at the alternative definition.

        \begin{theorem}
            Let $X$ be a smooth projective variety over an algebraically closed field $k$. Then the map $CH_0(X)_{\deg=0} \to Alb(X)$ induces an isomorphism on torsion prime to the characteristic of $k$.
        \end{theorem}

        \begin{proof}
            see \cite[Theorem 5.1]{SBloch}
        \end{proof}

        \begin{corollary}
            The group $CH_0(X)_{\deg=0}$ is representable if and only if $alb_X:CH_0(X)_{\deg=0} \to Alb(X)$ is an isomorphism.
        \end{corollary}
        
        \begin{lemma}\label{inverse}
            The fiber $(\theta_d^X)^{-1}(0)$ of the map $\theta_d^X:Sym^d(X) \times Sym^d(X) \to CH_0(X)_{\deg=0}$ is a countable union of closed algebraic subsets of $Sym^d(X) \times Sym^d(X)$.
        \end{lemma}
        
        \begin{proof}
            For a proof of this statement see \cite[Lemma 3]{Mum}, or \cite[Lemma 10.7]{CVoi2} over $\mathbb{C}$. It remains true in our scenario since by \autoref{Wdcclosed} we have that $W^d{d,0}$ is c-closed and the below equalities,
            \[
                \begin{array}{cc}
                    ({\theta_d^X})^{-1}(0) & = \{(A, B) \in Sym^{d,d}(X) : \theta_d^X(A, B) = 0\} \\
                     & = \{(A, B) \in Sym^{d,d}(X) : A - B \sim_{rat} 0\} \\
                     & = \{(A, B) \in Sym^{d,d}(X) : A \sim_{rat} B\} \\
                     & = W^{d, 0}.
                \end{array}
            \]
        \end{proof}

\section{A theorem on the Gysin kernel}
    \subsection{Part (a)}
        \begin{theorem}\label{parta}
            For every $t_0 \in U$, there is an abelian subvariety $A_{t_0}$ of $B_{t_0} \subset J_{t_0}$ such that
            \[
                G_{t_0} = \bigcup_{\text{countable}} \text{translates of }A_{t_0}.
            \]
        \end{theorem}

        \begin{proof}
            For each natural number $d$, let $Sym^d(C_{t_0})$ be the $d$-th symmetric product of the curve, $Sym^d(S_0)$ the $d$-th symmetric product of the surface ${S}_0$, $Sym^d(r_{t_0}):Sym^d(C_{t_0}) \to Sym^d({S}_0)$ the morphism from the $d$-th symmetric product of the curve $C_{t_0}$ to the $d$-th symmetric product of the surface ${S}_0$, induced by the closed embedding $r_{t_0}:C_{t_0} \hookrightarrow {S}_0$, and
            \[ Sym^{d,d}(r_{t_0}):Sym^{d,d}(C_{t_0}) = Sym^d(C_{t_0}) \times Sym^d(C_{t_0}) \to Sym^{d,d}(S_0) = Sym^d({S}_0) \times Sym^d({S}_0). \]

            We obtain the following commutative diagram
                \[
                    \begin{tikzcd}
                        Sym^{d,d}(C_{t_0}) \arrow[r, "Sym^{d,d}(r_{t_0})"] \arrow[d, "\theta_d^{C_{t_0}}"] & Sym^{d,d}(S_0) \arrow[d, "\theta_d^{{S}_0}"] \\
                        CH_0(C_{t_0})_{\deg=0} \arrow[r, "r_{t_0*}"] \arrow[d, "alb_{C_{t_{0}}}", "\sim"'] & CH_0(S_0)_{\deg=0} \arrow[d, "alb_{{S}_{0}}"] \\
                        Alb(C_{t_{0}}) = J_{t_{0}} \arrow[r, "Alb(r_{t_{0}*})"] & Alb(S_{0}) = J_{{S}_{0}}
                    \end{tikzcd}
                  \]
            
            We know that the Albanese map $alb_{C_{t_{0}}}:CH_{0}(C_{t_{0}})_{\text{deg}=0} \to Alb(C_{t_{0}}) = J_{t_{0}}$ is an isomorphism (see \cite[Lemma 3.46]{PauScho}). Thus $CH(C_{t_{0}})_{\text{deg}=0}$ is representable, i.e. $\theta^{C_{t_{0}}}_{d}$ is surjective for $d \gg 0$. This implies that the Gysin kernel is of the form $G_{t_{0}} = \theta^{C_{t_{0}}}_{d}( \theta^{S_{0}}_{d} \circ Sym^{d,d}(r_{t_{0}}))^{-1}(0) $.

            Since by \autoref{inverse} $(\theta^{S_{0}}_{d})^{-1}(0)$ is a countable union of Zariski closed subsets in $Sym^{d,d}({S}_{0}) $ it follows that $(\theta^{{S}_{0}}_{d} \circ Sym^{d,d}(r_{t_{0}})) ^{-1}(0) $ is a countable union of Zariski closed subsets in $Sym^{d,d}(C_{t_{0}})$. Thus $alb_{C_{t_{0}}}(G_{t}) = alb_{C_{t_{0}}} \circ \theta^{C_{t_{0}}}_{d} [(\theta^{{S}_{0}}_{d} \circ Sym^{d,d}(r_{t_{0}}))^{-1}(0)] $ is a countable union of Zariski closed subsets in the abelian variety $J_{t_{0}}$.

            Since $alb_{C_{t_{0}}}(G_{t_{0}})$ is over an uncountable field, it admits an unique irredundant decomposition inside $J_{t_{0}}$ (see \autoref{uniqueIrred}). We identify $alb_{G_{t_{0}}}$ with $G_{t_{0}}$ and write $G_{t_{0}} = \bigcup_{n \in \mathbb{N}}(G_{t_{0}})_{n}$ for the irredundant decomposition of $G_{t_{0}}$ in $J_{t_{0}} \cong CH_{0}(C_{t_{0}})_{\text{deg}=0}$.

            On the other hand, note that by definition $G_{t_{0}}$ is a subgroup in $CH_{0}(C_{t_{0}})_{\text{deg}=0}$, hence its image $alb_{C_{t_0}}(G_{t_{0}})$ in $J_{t_{0}}$ via $alb_{C_{t_{0}}}$ is also a subgroup in $G_{t_{0}}$.

            Thus the irredundant decomposition of $G_{t_{0}}$ contains a unique irreducible component passing through $0$ which is an abelian subvariety of $J_{t_{0}}$ (see \autoref{irredundantAbelian}). After renumbering the components, we may assume that this component is $(G_{t_{0}})_{0}$.

            It is clear that for any $x \in G_{t_0}$, the set $x + (G_{t_0})_0$ is an irreducible Zariski closed subset in $G_{t_0}$, and that we can write $G_{t_0} = \bigcup_{x \in G_{t_0}} (x + (G_{t_0})_0)$.

            Ignoring each set $x + (G_{t_{0}})_{0} \subset y + (G_{t_{0}})_{0}$ for $x,y \in G_{t_{0}}$, we get a subset $\Xi \subset G_{t_{0}}$ such that $G_{t_{0}} = \cup_{x \in \Xi} (x + (G_{t_{0}})_{0})$ which is an irredundant decomposition of $G_{t_{0}}$.

            For any $x, y \in \Xi, x + (G_{t_0})_0$ and $y + (G_{t_0})_0$ are irreducible and not contained in one another. Since $x + (G_{t_0})_0$ is irreducible and $G_{t_0}$ is a subgroup, then $x + (G_{t_0})_0 \subset (G_{t_0})_n$ for some $n$, for otherwise creating a contradiction to \autoref{CountContradiction}. It follows that $(G_{t_0})_0 \subset -x + (G_{t_0})_n$. Similarly, we can prove that $-x + (G_{t_0})_n$ is contained in $(G_{t_0})_l$ for some $l \in \mathbb{N}$. Then $(G_{t_0})_0 = (G_{t_0})_l$, that is, $(G_{t_0})_0 \subset -x + (G_{t_0})_n \subset (G_{t_0})_0$, so $-x + (G_{t_0})_n = (G_{t_0})_0$, i.e. $x + (G_{t_0})_0 = (G_{t_0})_n$ for each $x \in \Xi$. It means that $\Xi$ is countable.

            Taking $A_{t_0} = (G_{t_0})_0$, until now we have proved that there is an Abelian variety $A_{t_0} \subset J_{t_0}$ such that $G_{t_0} = \bigcup_{x \in \Xi} (x + A_{t_0})$, where $\Xi \subset G_{t_0}$ is a countable subset. Equivalently, we can write as follows: there is an Abelian variety $A_{t_0} \subset J_{t_0}$ such that $G_{t_0} = \bigcup_{\text{countable}} \text{ translates of } A_{t_0}$.
            
            In \cite{PauScho}, the authors prove using Hodge theory that
            $H^1(A_{t_0,\mathbb{C}},\mathbb{Z}) \subset H^1(B_{t_0,\mathbb{C}},\mathbb{Z})$,
            from which they conclude that $A_{t_0} \subset B_{t_0}$.
            By Artin’s comparison theorem \cite[Theorem 21.1]{JMilLec}, the first inclusion may be translated into an inclusion of étale cohomology groups of the corresponding complex varieties. Using a lift to characteristic zero together with the specialization comparison theorem \cite[Theorem 20.5]{JMilLec}, we obtain an inclusion of Galois representations
            \[
                H^1_{\acute{e}t}(A_{t_0,\overline{k}},\mathbb{Q}_\ell)
                \subset
                H^1_{\acute{e}t}(B_{t_0,\overline{k}},\mathbb{Q}_\ell).
            \]

            By Tate’s isogeny theorem for abelian varieties over finitely generated fields in positive characteristic \cite{Zarhin}, the inclusion above determines, via the identification of algebraic endomorphisms with $\ell$-adic endomorphisms, an idempotent in $\mathrm{End}(B_{t_0})\otimes\mathbb{Q}_\ell$. As explained in \cite[page 7]{BanGul}, such an idempotent corresponds, up to isogeny, to an abelian subvariety of $B_{t_0}$ whose $\ell$-adic cohomology realizes the given subrepresentation. Since the above inclusion is compatible with the embeddings $A_{t_0} \subset J_{t_0}, B_{t_0}\subset J_{t_0}$, this abelian subvariety coincides with $A_{t_0}$, and therefore $A_{t_0}\subset B_{t_0}$.

        \end{proof}


    \subsection{Part (b)}
        From now on, we will only be working over an uncountable algebraically closed field $k$ of characteristic $p$ and thus we do not need to worry about the lift. Thus, from this point on, we will drop our notation of subfixing with $0, \eta$ and $\overline{\eta}$ from \autoref{sec:LF} and when using a subfix, it will now mean the fiber of a point over the base scheme as seen below.
        
        \subsubsection{Connection between generic geometric point and very general points}{\label{part.b.connection}}
            Let $k$ be an uncountable algebraically closed field of characteristic $p$. Let $T$ be an integral scheme over $k, \; \mathfrak{X}_T$ a scheme over $T$ and $X_t = f_T^{-1}(t)$ the fiber over $t \in T$ of the flat family $f_T:\mathfrak{X}_t \to T$.

            \begin{lemma}
                \label{ggp}
                Given an integral base $T$ over $k$ there exists a natural c-open subset $U_c$ in $T$ such that every $t \in U_c$ is scheme-theoretic isomorphic to the generic geometric point $\overline{\eta}$ of $T$ and given a flat family $f_T:\mathfrak{X}_T \to T$ over $T$, the above scheme-theoretic isomorphism of points induces isomorphisms between the fibers $X_t$ for all closed points $t \in U_c$, and the geometric generic fiber $X_{\overline{\eta}}$, as schemes over $Spec(\mathbb{F}_p)$, moreover these isomorphisms preserve rational equivalence of algebraic cycles.
            \end{lemma}

            \begin{proof}
                The proof can be found in \cite[Lemma 5.4]{PauScho} except that we work over the primary subfield $\mathbb{F}_{p}$ instead of $\mathbb{Q}$. However, the proof is still true due to 
                the fact that the transcedental degree $[k : \mathbb{F}_{p}]$ is infinite.
            \end{proof}

            Thus, for any integral scheme $T$ over $k$ and for any morphism $T \to (\mathbb{P}^{d})^{\vee}$, let $f_{T}:C_{T}\to T$ be the family of hyperplane sections of $S$ (remember $S = S_{0}$) parameterized by $T$, $g_{T}:S_{T} \to T$ the family such that each fiber over $T$ is isomorphic to $S$ and

           \[
               \begin{tikzcd}
                 C_{T} \arrow[rr, "r_{T}"] \arrow[dr, "f_{T}"] & & S_{T} \arrow[dl, "g_{T}"] \\
                                                              & T &
               \end{tikzcd}
           \]
          
           the closed embeddings of schemes over $T$. Then we also have the closed embedding $r_{0}$ and $r_{\overline{\eta}}$ over $\{0\} = Spec(k)$ and the generic geometric point $\{\overline{\eta}\} = Spec(\overline{k(T)})$, where $k(T)$ is the function field of T.

           By \autoref{ggp} above, there exists a natural c-open subset $U_c$ in $T$ such that the residue field of any closed point in $U_c$ is isomorphic to the residue field of the geometric generic point of $T$. Thus we get the following commutative diagram

          \[
              \begin{tikzcd}
                   & Sym^{d,d}(C_{t}) \arrow[r, "Sym^{d,d}(\kappa_{t}^{f_{T}})"] \arrow[d, "\theta_{d}^{C_{t}}"] & Sym^{d,d}(C_{\overline{\eta}}) \arrow[d, "\theta_{d}^{C_{\overline{\eta}}}"] & \\
                   & CH_{0}(C_{t})_{\text{deg}=0} \arrow[r, "\kappa_{t*}^{f_{T}}", "\sim"'] \arrow[d, "alb_{C_{t}}", "\sim"'] \arrow[ddl, "r_{t*}"] & CH_{0}(C_{\overline{\eta}})_{\text{deg}=0} \arrow[d, "alb_{C_{\overline{\eta}}}", "\sim"'] \arrow[ddr, "r_{\overline{\eta}*}"] & \\
                   & J_{t} \arrow[r, "l_{t}"] & J_{\overline{\eta}} & \\
                CH_{0}(S_{t})_{\text{deg}=0} \arrow[rrr, "\kappa_{t*}^{g_{T}}", "\sim"'] & & &  CH_{0}({S}_{\overline{\eta}})_{\text{deg}=0}
              \end{tikzcd}
          \]
          where for every $t \in U_c$, $\kappa_{t}^{f_{T}}$ and $\kappa_{t}^{g_{T}}$ are isomorphisms as schemes over $Spec(\mathbb{F}_{p})$. We define $l_{t} = alb_{C_{\overline{\eta}}} \circ \kappa_{t*}^{f_{T}} \circ alb_{{C_{t}}}^{-1}$ which is a regular morphism.
    
            \begin{lemma}\label{ell}
                For any closed point $t \in U_c,$ we have $l_{t}(B_{t}) = B_{\overline{\eta}}$ and $l_{t}(A_{t}) = A_{\overline{\eta}}$.
            \end{lemma}
    
            \begin{proof}
                See \cite[Proposition 20]{BanGul}.
            \end{proof}
    
            Thus, when studying the varieties $A_{t}$ in a family, we can either work at the level of the geometric generic point or at a very general closed point on the base scheme. Another good reference for the above section is \cite{vial}.
    
        \subsubsection{Relation between \texorpdfstring{$A_{t}$}{Lg} and \texorpdfstring{$B_{t}$}{Lg}}
            Let $L$ be a (very) ample line bundle on $J_{t}$ and $i:A_{t} \to J_{t}$ 
            the closed embedding. Let $L_{t}$ be the pull-back of $L$ to $A_{t}$ under the embedding $i$. Then via the following commutative diagrams, define the homomorphism $\zeta$ on divisors and similarly the homomorphism $\zeta_{\mathbb{Z}_{l}}$ on cohomology
            \[
                \begin{tikzcd}
                    CH_0(A_{t})_{\deg=0} \arrow[r, "\zeta"] \arrow[d, "(\lambda_{L_{t}})_*"] & CH_0(J_{t})_{\deg=0} \\
                    CH_0(A_{t}^\vee)_{\deg=0} \arrow[r, "i^{\vee*}"] & CH_0(J_{t}^\vee)_{\deg=0} \arrow[u, "\lambda_L^*"],
                \end{tikzcd}
                \begin{tikzcd}
                    H_{\acute{e}t}^1(A_{t}, \mathbb{Z}_{\ell}) \arrow[r, "\zeta_{\mathbb{Z}_{\ell}}"] \arrow[d, "(\lambda_{L_{t}})_*"] & H_{\acute{e}t}^1(J_{t}, \mathbb{Z}_{\ell}) \\
                    H_{\acute{e}t}^1(A_{t}^\vee, \mathbb{Z}_{\ell}) \arrow[r, "i^{\vee*}"] & H_{\acute{e}t}^1(J_{t}^\vee, \mathbb{Z}_{\ell}) \arrow[u, "\lambda_L^*"]
                \end{tikzcd}
            \]
            Analogously we do the same for $\zeta_{\mathbb{Q}_{\ell}/\mathbb{Z}_{\ell}}:H_{\acute{e}t}^1(A_{t}, \mathbb{Q}_{\ell}/\mathbb{Z}_{\ell}) \to H_{\acute{e}t}^1(J_{t},\mathbb{Q}_{\ell}/ \mathbb{Z}_{\ell})$.
            Thus, we get the injective homomorphisms $\zeta_{\mathbb{Q}_{\ell}}:H_{\acute{e}t}^1(A_{t}, \mathbb{Q}_{\ell}) = H_{\acute{e}t}^1(A_{t}, \mathbb{Z}_{\ell}) \otimes \mathbb{Q}_{\ell} \to H_{\acute{e}t}^1(J_{t}, \mathbb{Q}_{\ell}) = H_{\acute{e}t}^1(J_{t}, \mathbb{Z}_{\ell}) \otimes \mathbb{Q}_{\ell}$ and $\zeta_{\mathbb{Z}_{\ell}} \otimes \mathbb{Q}_{\ell}/\mathbb{Z}_{\ell}:H_{\acute{e}t}^1(A_{t}, \mathbb{Z}_{\ell}) \otimes \mathbb{Q}_{\ell} / \mathbb{Z}_{\ell} \to H_{\acute{e}t}^1(J_{t}, \mathbb{Z}_{\ell}) \otimes \mathbb{Q}_{\ell} / \mathbb{Z}_{\ell}$ induced by $\zeta_{\mathbb{Z}_{\ell}}$.

            \begin{proposition}
                The image of the composition
                \[
                    H_{\acute{e}t}^1(A_{t}, \mathbb{Q}_{\ell}(1-p)) \xrightarrow{\zeta_{\mathbb{Q}_{\ell}}} H_{\acute{e}t}^1(J_{t}, \mathbb{Q}_{\ell}(1-p)) \xrightarrow{w_*} H_{\acute{e}t}^{1}(C_{t}, \mathbb{Q}_{\ell})
                \]
                is contained in the kernel of the Gysin homomorphism
                \[
                    H_{\acute{e}t}^1(C_{t}, \mathbb{Q}_{\ell}) \xrightarrow{r_{{t}*}} H_{\acute{e}t}^3(S, \mathbb{Q}_{\ell})
                \]
            \end{proposition}

            \begin{proof}
                The proof can be found in \cite[Proposition 14]{BanGul},  where K. Banerjee and V. Guletski\u{i} 
                proved it over an uncountable algebraically closed field of characteristic zero. The proof is still viable in our new scenario over $k$ since the authors work with Bloch's $\ell$-adic Abel-Jacobi map and then with the canonical homomorphism
                \[
                    \lambda_\ell^p(V):CH^p(V)(\ell) \to H_{\acute{e}t}^p(V, \mathbb{Q}_\ell / \mathbb{Z}_\ell (p))
                \]
                where $V$ is a smooth projective variety over $k$ and $CH^p(V)(\ell)$ is the $\ell$-torsion subgroup (see \cite{SBloch79}). It is known that $\lambda_\ell^1$ is an isomorphism and the homomorphisms $\lambda_\ell^p$ are functorial with respect to the action of correspondences between smooth projective varieties over an algebraically closed field. By nature, $A_t$ and $J_t$ are autodual since they are principally polarized by the theta divisor and their N\'{e}ron-Severi groups are torsion free (see \cite[Corollary 10.18]{JMilAbl}). In \cite{GorGul12}, S. Gorchinskiy and   V. Guletski\u{i} 
                showed that the homomorphism
                \[
                    \rho_\ell^{i,j}:H_{\acute{e}t}^i(V, \mathbb{Z}_\ell(j)) \otimes \mathbb{Q}_\ell / \mathbb{Z}_\ell \to H_{\acute{e}t}^i(V, \mathbb{Q}_\ell / \mathbb{Z}_\ell (j))
                \]
                has finite kernel and cokernel over an algebraically closed field.
                
                The morphism of motives $\omega:M(J_t) \otimes \mathbb{L}^{p-1} \to M(C_t)$ induces the homomorphism $\omega_*:CH_0(J_t)_0 \to CH^p(C_t)_0$, where $M(-)$ is the functor from smooth projective varieties to (contravariant) Chow motives, $\mathbb{L}$ is the Lefschetz motive and $\mathbb{L}^n$ is the $n$-fold tensor product of $\mathbb{L}$. This induces the homomorphism $\omega_*:H_{\acute{e}t}^1(J_t, \mathbb{Q}_\ell(1-p)) \to H_{\acute{e}t}^{2p-1}(C_t, \mathbb{Q}_\ell)$.

                Composing all of these together gives us what we desire.
            \end{proof}

            Following the ideas from \cite[\S 6]{BanGul} (or more generally, see \cite[Chapter 3]{Freitag}) we consider the minimal subextension $L$ of $k(D)$ in $\overline{k(D)}$ where $D$ is a projective line in the dual space $(\mathbb{P}^d)^\vee$ such that the abelian varieties $A_{\overline{\eta}}, B_{\overline{\eta}}$ and $J_{\overline{\eta}}$ are defined over $L$. Additionally, we consider 
            $L = k(D')$ for an algebraic curve $D'$. Since now we have $A_{\overline{\eta}}, B_{\overline{\eta}}$ and $J_{\overline{\eta}}$ defined over $L$, their closed embeddings $A_{\overline{\eta}} \hookrightarrow B_{\overline{\eta}}$ and $B_{\overline{\eta}} \hookrightarrow J_{\overline{\eta}}$ are also defined over $L$. Thus, there exists a Zariski open subset $U'$ in $D'$ with spreads $\mc{A}_{\overline{\eta}}, \mc{B}_{\overline{\eta}}$ and $\mc{J}_{\overline{\eta}}$ of $A_{\overline{\eta}}, B_{\overline{\eta}}$ and $J_{\overline{\eta}}$ over $U'$. We then know that the (\'{e}tale) fundamental group $\pi_1(U', \overline{\eta})$ acts continuously on $H_{\acute{e}t}^1(\mc{A}_{\overline{\eta}}, \mathbb{Q}_{\ell}) = H_{\acute{e}t}^1(A_{\overline{\eta}}, \mathbb{Q}_{\ell})$, $H_{\acute{e}t}^1(\mc{B}_{\overline{\eta}}, \mathbb{Q}_{\ell}) = H_{\acute{e}t}^1(B_{\overline{\eta}}, \mathbb{Q}_{\ell})$ and $H_{\acute{e}t}^1(\mc{J}_{\overline{\eta}}, \mathbb{Q}_{\ell}) = H_{\acute{e}t}^1(J_{\overline{\eta}}, \mathbb{Q}_{\ell})$ (see \cite[Appendix 1.8 Proposition]{Freitag}).

            Similarly, for each closed point $s \in D \setminus U$, let $\delta_s \in H_{\acute{e}t}^1(C_{\overline{\eta}}, \mathbb{Q}_{\ell})$ be the vanishing cycle, unique up to conjugation, corresponding to the point $s$ and let $E \subset H_{\acute{e}t}^1(C_{\overline{\eta}}, \mathbb{Q}_{\ell})$ be the $\mathbb{Q}_{\ell}$-vector subspace generated by all the elements $\delta_s$. Note that the monodromy action of the tame fundamental group on the space of vanishing cycles $E$ is irreducible \cite[Chapter 3, Corollary 7.4]{Freitag}. Recall that in this scenario, the Picard-Lefschetz formula says $\gamma_s(u)x = x \pm \overline{u}\langle x, \delta_s \rangle \delta_s$ (see \cite[Chapter 3 Theorem 7.1]{Freitag}).

            \begin{lemma}
                Under the assumptions above, either $A_{\overline{\eta}} = 0$ or $A_{\overline{\eta}} = B_{\overline{\eta}}$.
            \end{lemma}
            
            \begin{proof}
                By \cite[Proposition 15]{BanGul} and the identification of the space of vanishing cycles with the kernel of the Gysin morphism
                \[
                    r_{\bar\eta *}: H^1_{\acute et}(C_{\bar\eta},\mathbb{Q}_\ell) \to H^3_{\acute et}(S_{\bar\eta},\mathbb{Q}_\ell)
                \]
                (see \cite[\S 4.3]{Deligne}), the image of the composition
                \[
                    H^1_{\acute et}(A_{\bar\eta},\mathbb{Q}_\ell(1-p)) \xrightarrow{\;\zeta_{\mathbb{Q}_\ell}\;} H^1_{\acute et}(J_{\bar\eta},\mathbb{Q}_\ell(1-p)) \xrightarrow{\;\omega_*\;} H^3_{\acute et}(S_{\bar\eta},\mathbb{Q}_\ell)
                \]
                is contained in the space of vanishing cycles \(E\).

                The morphism \(\zeta_{\mathbb{Q}_\ell}\) is injective and compatible with the action of the fundamental group
                \(\pi_1(U',\bar\eta)\), while \(\omega_*\) is an isomorphism
                (see \cite[Remark 4]{BanGul}). Hence the image
                \[
                    E_0 := Im(\omega_* \circ \zeta_{\mathbb{Q}_\ell})
                \]
                is a $\pi_1(U',\bar\eta)$-stable subspace of \(E\).

                By spreading out the correspondences inducing $\omega_*$ and the morphisms defining $\zeta_{\mathbb{Q}_\ell}$ over a suitable Zariski open subset of the base, the resulting action is compatible with the monodromy action of the tame fundamental group $\pi_1^t(U,\bar\eta)$ (see \cite[\S 2]{BanGul}).

                According to \cite[Chapter 3, Corollary 7.4]{Freitag}, the vanishing cycles representation $E$ is irreducible as a $\pi_1^t(U,\bar\eta)$-module.

                After replacing the base field by a finitely generated subfield and choosing $\ell \neq \operatorname{char}(k)$, the Tate–Zarhin theorem \cite{Zarhin} implies that $H^1_{\acute et}(A_{\bar\eta},\mathbb{Q}_\ell)$ is a semisimple $\pi_1(U',\bar\eta)$-representation. Consequently, its image $E_0 \subset E$ is a semisimple $\pi_1^t(U,\bar\eta)$-subrepresentation.

                Since $E$ is irreducible as a $\pi_1^t(U,\bar\eta)$-module, it follows that either $E_0 = 0$ or $E_0 = E$.

                If $E_0 = 0$, then $H^1_{\acute et}(A_{\bar\eta},\mathbb{Q}_\ell) = 0$. By Tate's isogeny theorem \cite{Tate} and \cite{Zarhin} this implies that $\dim A_{\bar\eta} = 0$, and hence $A_{\bar\eta} = 0$.

                If $E_0 = E$, then
                \[
                    H^1_{\acute et}(A_{\bar\eta},\mathbb{Q}_\ell(1-p))
                    \cong
                    H^1_{\acute et}(B_{\bar\eta},\mathbb{Q}_\ell(1-p))
                \]
                as $\pi_1(U',\bar\eta)$-representations. Again by the Tate isogeny theorem \cite{Tate} and \cite{Zarhin} this induces an isogeny between $A_{\bar\eta}$ and $B_{\bar\eta}$. Since both are abelian subvarieties of $J_{\bar\eta}$ and their $\ell$-adic cohomology realizes the same subrepresentation of $H^1_{\acute et}(J_{\bar\eta},\mathbb{Q}_\ell)$, this isogeny is induced by the identity on $J_{\bar\eta}$, and therefore $A_{\bar\eta} = B_{\bar\eta}$.
            \end{proof}
          
\bigskip
            
            Assume that $U$ is the base integral scheme $T$ over $k$ introduced in subsection \ref{part.b.connection}. That is, we will be studying the global case. Again, let $U_c$ be a $c$-open subset 
            in $U=T$ by \autoref{ggp}. In the dual space $(\mathbb{P}^d)^\vee$, if $\xi$ is the generic point and $\overline{\xi}$ the corresponding geometric generic point, then for any closed point $P \in U_c$ one has the isomorphism $C_P \cong C_{\overline{\xi}}$ and for any two closed points $P, P' \in U_c$, one has the scheme-theoretic isomorphism $C_P \cong C_{P'}$.
    
            We are now ready 
            to prove part (b) of a theorem on the Gysin kernel.
    
            \begin{theorem}
                For very general $t \in U$, either $A_{t} = 0$ or $A_{t} = B_{t}$, i.e. there exists a c-open subset $U_c \subset U$ such that either $A_{\overline{\eta}} = 0$, in which case $A_{t} = 0$ for all $t \in U_c$, or $A_{\overline{\eta}} = B_{\overline{\eta}}$, in which case $A_{t} = B_{t}$ for all $t \in U_c$.
            \end{theorem}
    
            \begin{proof}
                Recall that in our setting we can always work with a Lefschetz pencil (see \cite[Chapter 3 Proposition 1.10]{Freitag}). Thus we can choose a line $D \in Z = (\mathbb{P}^{d})^{\vee} \setminus U_{0}$ giving a Lefschetz pencil $f_{D}$, where $D \cap U_{c} \neq 0$. Let $P_0 \in D \cap U_c$ be a point and $\overline{\eta}$ be the geometric generic point of $D$. By the Proposition above, we have either $A_{\overline{\eta}} = 0$ or $A_{\overline{\eta}} = B_{\overline{\eta}}$.

                Suppose that $A_{\overline{\eta}} = 0$. Then applying \autoref{ell} on the pencil $f_D$, we get that $A_{P_0} = 0$. Applying \autoref{ell} to the family $f_T$, we obtain  $A_{\overline{\xi}} = 0$ and so for each closed point $P \in U_c$, the abelian variety $A_P = 0$.

                Suppose that $A_{\overline{\eta}} = B_{\overline{\eta}}$. Then similarly applying \autoref{ell} to $f_D$ and the family $f_T$ we get $A_{P_0} = B_{P_0}$ and $A_{\overline{\xi}} = B_{\overline{\xi}}$, respectively. Thus, for each closed point $P \in U_c$, we obtain  $A_P = B_P$.
            \end{proof}

\bigskip


Claudia Schoemann, Mathematisches Seminar der Christian-Albrechts-Universität zu Kiel, 24118 Kiel, Germany; Simion Stoilow Institute of Mathematics of the Romanian Academy (IMAR), Bucharest, Romania, e-mail: schoemann@math.uni-kiel.de

\medskip

Skylar Werner, Mathematisches Institut der Universität Göttingen, Bunsenstrasse 3-5, 37073 Göttingen, Germany, email: Skylar.Werner@stud.uni-goettingen.de


\begin{thebibliography}{10}

%
\bibitem{BanGul}
K. Banerjee and V. Guletski\u{i}.
\newblock {\em \'{E}tale monodromy and rational equivalence for 1-cycles on cubic hypersufraces in $\mathbb{P}^5$}.
\newblock Mat. Sb., 211(2):3--45, 2020.
%
\bibitem{BarEkl}
J. Barwise and P. Eklof.
\newblock {\em Lefschetz's principle}.
\newblock {Journal of Algebra}, 13:554 -- 570, 1969.
%
\bibitem{BloAlgK}
S. Bloch.
\newblock {\em Algebraic cycles and higher k-theory}.
\newblock {Advances in Mathematics}, 61:267 -- 304, 1986.
%
\bibitem{SBloch}
S. Bloch.
\newblock {\em Lectures on Algebraic Cycles}. 2nd ed., New Mathematical Monographs, vol. 16, Cambridge University Press, 2010.
%
\bibitem{SBloch79}
S. Bloch.
\newblock {\em Torsion algebraic cycles and a theorem of Roitman}. Compositio Math. vol. 39, no. 1, 107 - 127, 1979.
%
\bibitem{Deligne}
P. Deligne.
\newblock {\em La conjecture de {W}eil {II}}.
\newblock {Publication Math\'{e}matiques de l'IH\'{E}S}, 52:137 -- 252, 1980.
%
%
%
%
\bibitem{Freitag}
E. Freitag and R. Kiehl.
\newblock {\em Etale Cohomology and the Weil Conjecture}. volume~13 of Ergebnisse der Mathematik und ihrer Grenzgebiete. 3. Folge / A Series of Modern Surveys in Mathematics.
\newblock Springer-Verlag, 1988.
%
\bibitem{FulInt}
W. Fulton.
\newblock {\em Intersection Theory}.
\newblock A Series of Modern Surveys in Mathematics. Springer-Verlag, 1984.
%
%
\bibitem{GorGul12}
S. Gorchinskiy, V. Guletski\u{i}.
\newblock {\em Motives and representability of algebraic cycles on three-folds over a field}.
\newblock J. Algebraic Geom. 21, no. 2, 347 - 373, 2012.
%
\bibitem{EGAIII}
A. Grothendieck.
\newblock {\em \'El\'ements de g\'eom\'etrie alg\'ebrique. III. \'Etude cohomologique des faisceaux coh\'erents. I}.
\newblock Publications Math\'ematiques de l'IH\'ES, No. 11, 1961.
%
%
%
\bibitem{SGAI}
A. Grothendieck.
\newblock {\em S\'{e}minaire de G\'{e}ometrie Alg\'{e}brique Rev\^{e}tements \'{e}tales et groupe fondamental (1960 - 61)}, volume 224 of Lecture Notes in Math.
\newblock Springer, 1971.
%
%
%
\bibitem{Hart}
R. Hartshorne.
\newblock {\em Algebraic Geometry}, volume~52 of {Graduate Texts in Mathematics}.
\newblock Springer, 1977.
%
%
\bibitem{RuledSurf}
A. Mattuck.
\newblock {\em Ruled {S}urfaces and the {A}lbanese mapping}.
\newblock {Proceedings of the American Mathematical Society}, 1969.
%
%
\bibitem{JMilAbl}
J.S. Milne.
\newblock {\em Abelian Varieties (v2.00)}.
\newblock Available at \url{www.jmilne.org/math/}, 2008.
%
\bibitem{JMilEtal}
J.S. Milne.
\newblock {\em \'{E}tale Cohomology}.
\newblock Princeton University Press, Princeton, New Jersey, 1980.
%
\bibitem{JMilLec}
J.S. Milne.
\newblock {\em Lectures on \'{E}tale Cohomology (v2.21)}.
\newblock Available at \url{www.jmilne.org/math/}, 2013.
%
\bibitem{MilRoit}
J.S. Milne.
\newblock {\em Zero cycles on algebraic varieties in nonzero characteristic: Ro\u{i}tman's theorem}.
\newblock {Compositio Mathematica}, 47:271 -- 287, 1982.
%
%
%
%
\bibitem{Mum}
D. Mumford.
\newblock {\em Rational equivalence of 0-cycles on surfaces}.
\newblock Journal of mathematics of Kyoto University
\newblock 9(2):195-204, 1969.
%
%
\bibitem{ENoeth}
E. Noether.
\newblock {\em Der {E}ndlichkeitssatz der {I}nvarianten endlicher linearer {G}ruppen der {C}harakteristik p}.
\newblock Nachrichten von der Gesellschaft der Wissenschaften zu G\"{o}ttingen, Mathematisch-Physikalische Klasse 1926, pp. 28 - 35.
%
\bibitem{PauScho}
R. Paucar and C. Schoemann.
\newblock {\em On the {K}ernel of the {G}ysin {H}omomorphism on {C}how {G}roups of {Z}ero {C}ycles}.
\newblock Publications math\'{e}matiques de Besan\c{c}on: Alg\`{e}bre et Th\'{e}orie des Nombres, pp. 59 -- 104, 2024.
%
\bibitem{RoitGamma}
A.A. Ro\u{i}tman.
\newblock {\em On $\Gamma$-equivalence of zero-dimensional cycles}.
\newblock {\em Mathematics of the USSR-Sbornik}, 15(4):555 -- 567, 1971.
%
\bibitem{RoitRat}
A.A. Ro\u{i}tman.
\newblock {\em Rational equivalence of zero-dimensional cycles}.
\newblock {\em Mat. Sb. (N.S.)}, 89(131):569 -- 585, 1972.
%
\bibitem{Sernesi}
E. Sernesi.
\newblock {\em Deformations of Algebraic Schemes}, volume 334 of {\em Grundlehren der mathematischen Wissenschaften}.
\newblock Springer-Verlag, Berlin, 2006.
%
\bibitem{SeLocal}
J.P. Serre.
\newblock {\em Local Fields}. volume~67 of Graduate Texts in Mathematics.
\newblock Springer, 1979.
%
%
\bibitem{Severi}
F. Severi.
\newblock {\em Probl\`{e}mes r\'{e}solus et probl\`{e}mes nouveaux dans las th\`{e}orie des systemes d'\'{e}quivalence}.
\newblock {Proc. Internat. Cong. Math}, 3, 1954.
%
\bibitem{SilArith}
J.H. Silverman.
\newblock {\em The Arithmetic of Elliptic Curves}, volume 106 of {Graduate Texts in Mathematics}.
\newblock Springer, second edition, 2009.
%
%
\bibitem{Tate}
J. Tate.
\newblock {\em Endomorphisms of Abelian Varieties over Finite Fields}.
\newblock {Inventiones math}. 2, 134 - 144, 1966.
%
\bibitem{vial}
C. Vial.
\newblock {\em Algebraic cycles and fibrations}.
\newblock {Doc. Math}. 18, 1521 - 1553, 2013.
%
\bibitem{CVoi2}
C. Voisin.
\newblock {\em Hodge Theory and Complex Algebraic Geometry II}.
\newblock volume~77 of Cambridge Studies in Advanced Mathematics.
\newblock Cambridge University Press, Cambridge, 2003.
\newblock Translated from the French by Leila Schneps.
%
\bibitem{LefPrin}
A. Weil.
\newblock {\em Foundations of Algebraic Geometry}.
\newblock American Mathematical Society, 1962.
\newblock Revised and Enlarged Edition.
%
\bibitem{Zarhin}
J.G. Zarhin.
\newblock {\em Endomorphisms of Abelian Varieties over fields of finite characteristic}.
\newblock Mathematics of the USSR - Izvestiya, Volume 9, Issue 2, 255 - 260 
(1975).

\end{thebibliography}
\end{document}